\documentclass[a4paper,12pt]{article}
\usepackage{amsmath}
\usepackage{amssymb}
\usepackage{setspace}
\usepackage{fullpage}
\usepackage{bbm}
\usepackage{color}
\usepackage{amsthm}
\newtheorem{theorem}{Theorem}[section]
\newtheorem{lemma}[theorem]{Lemma}

\newtheorem{corollary}[theorem]{Corollary}

\theoremstyle{definition}

\newtheorem*{definition}{Definition}

\usepackage{pdfsync}
\def\PG{\mathrm{PG}}

\def\A{\mathcal{A}}  
\def\D{\mathcal{D}} \def\E{\mathcal{E}} 
 \def\K{\mathcal{K}}
\def\L{\mathcal{L}}  
\def\O{\mathcal{O}} \def\P{\mathcal{P}}
\def\R{\mathcal{R}} \def\S{\mathcal{S}}
\def\T{\mathcal{T}}

\def\F{\mathbb{F}}

\title{Characterisations of elementary pseudo-caps and good eggs}
\author{ Sara Rottey \and Geertrui Van de Voorde\thanks{This author is supported by the Fund for Scientific Research Flanders
(FWO -- Vlaanderen).}}
\begin{document}
\maketitle
\begin{abstract} In this note, we use the theory of Desarguesian spreads to investigate good eggs. Thas showed that an egg in $\PG(4n-1,q)$, $q$ odd, with two good elements is elementary. By a short combinatorial argument, we show that a similar statement holds for large pseudo-caps, in odd and even characteristic. As a corollary, this improves and extends the result of Thas, Thas and Van Maldeghem (2006) where one needs at least $4$ good elements of an egg in even characteristic to obtain the same conclusion. We rephrase this corollary to obtain a characterisation of the generalised quadrangle $T_3(\O)$ of Tits.

Lavrauw (2005) characterises elementary eggs in odd characteristic as those good eggs containing a space that contains at least $5$ elements of the egg, but not the good element. We provide an adaptation of this characterisation for weak eggs in odd and even characteristic. As a corollary, we obtain a direct geometric proof for the theorem of Lavrauw.
\end{abstract}

\section{Preliminaries}
In this note, we study {\em eggs} and {\em pseudo-caps} in the projective space $\PG(n,q)$, where $\PG(n,q)$ denotes the $n$-dimensional projective space over the finite field $\F_q$ with $q$ elements, $q=p^h$, $p$ prime. Many previous proofs and characterisations of eggs rely on the connection with eggs and {\em translation generalised quadrangles} \cite{TGQ}. It is our aim to study eggs from a purely geometric perspective, without using this connection or coordinates. In Section \ref{2goodDes} we obtain a connection between good eggs and Desarguesian spreads. This link will enable us to reprove, improve or extend known results in Sections \ref{3twogood} and \ref{4ThmLavrauw}. We begin by repeating some well-known definitions.

\begin{definition}
A {\em cap} in $\PG(n, q)$ is a set of points such that every three points span a plane. A cap of size $k$ is denoted as a {\em $k$-cap}.
\end{definition}

A $k$-cap of $\PG(2,q)$ is often called a {\em $k$-arc}. A $k$-arc in $\PG(2,q)$ satisfies $k\leq q+1$ for $q$ odd and $k\leq q+2$ for $q$ even. A $(q+1)$-arc is called an {\em oval}, a $(q+2)$-arc a {\em hyperoval}.
A $k$-cap of $\PG(3,q)$, $q>2$ satisfies $k\leq q^2+1$, moreover,  a $(q^2+1)$-cap of $\PG(3,q)$ is often called an {\em ovoid}. We will consider the higher dimensional equivalent of these structures.

\begin{definition}
A {\em pseudo-cap} is a set $\mathcal{C}$ of $(n-1)$-spaces in $\PG(2n+m-1,q)$ such that any three elements of $\mathcal{C}$ span a $(3n-1)$-space.
\end{definition}

If $m=n$, a pseudo-cap is often called a {\em pseudo-arc}. By \cite{Thas PHO}, a pseudo-arc $\A$ in $\PG(3n-1,q)$ satisfies $|\A| \leq q^n+1$ for $q$ odd and $|\A| \leq q^n+2$ for $q$ even. If a pseudo-arc $\A$ has $q^n+1$ or $q^n+2$ elements, $\A$ is a {\em pseudo-oval} or {\em pseudo-hyperoval} respectively. If $m=2n$, a pseudo-cap with $q^{2n}+1$ elements is called a {\em pseudo-ovoid}.

Examples of pseudo-caps in $\PG(kn-1,q)$ arise by applying field reduction to caps in $\PG(k-1,q^n)$ and if a pseudo-cap is obtained by field reduction, we call it {\em elementary}. {\em Field reduction} is the concept where a point in $\PG(k-1,q^n)$ corresponds in a natural way to an $(n-1)$-space of $\PG(kn-1,q)$. The set of all points of $\PG(k-1,q^n)$ then correspond to a set of disjoint $(n-1)$-spaces partitioning $\PG(kn-1,q)$, forming a {\em Desarguesian spread}.
Every Desarguesian spread $\D$ has the property that the space spanned by $2$ elements of $\D$ is partitioned by elements of $\D$, i.e. $\D$ is {\em normal}. Moreover, a normal $(n-1)$-spread of $\PG(kn-1,q)$, $k>2$, is Desarguesian, see \cite{BarlottiCofman}.
For more information on field reduction and Desarguesian spreads we refer to \cite{FQ11}.

A {\em partial spread} in $\PG(n+m-1,q)$ is a set of mutually disjoint $(n-1)$-spaces.
Every element $E_i$ of a pseudo-cap $\E$ of $\PG(2n+m-1,q)$ defines a partial spread $$\S_{i}:=\{ E_0,\ldots,E_{i-1},E_{i+1},\ldots,E_{|\E|}\}/E_i$$ in $\PG(n+m-1,q)\cong \PG(2n+m-1,q)/E_i$ and we say that the element $E_i$ {\em induces} the partial spread $\S_i$.

A partial spread of $\PG(2n-1,q)$ of size $q^n$ is said to have {\em deficiency 1}.
From \cite{Beutelspacher}, we know that a partial spread of $\PG(2n-1,q)$ with deficiency 1 can be extended to a spread in a unique way, i.e. the set of points in $\PG(2n-1,q)$ not lying on an element of the partial spread, form an $(n-1)$-space.

\begin{definition}
A {\em weak egg} in $\PG(2n+m-1,q)$ is a pseudo-cap of size $q^{m}+1$.
\end{definition}
Clearly, pseudo-ovals and pseudo-ovoids are examples of weak eggs. A weak egg $\E$ in $\PG(2n+m-1,q)$ is called an {\em egg} if each element $E \in \E$ is contained in a $(n+m-1)$-space, $T_E$, which is skew from every element of $\E$ different from $E$. The space $T_E$ is called the {\em tangent space} of $\E$ at $E$. It is not hard to show that if $n=m$, then every weak egg is an egg. Eggs are studied mostly because of their one-to-one correspondance with {\em translation generalised quadrangles} of order $(q^n,q^{2n})$, see Subsection \ref{3.2 TGQ}.

The only known examples of eggs in $\PG(2n+m-1,q)$ have either $m=n$ or $m=2n$ and we have the following theorem restricting the number of possibilities for the parameters $n$ and $m$.

\begin{theorem}{\rm \cite[Theorem 8.7.2]{FGQ}}\label{mogelijkheden}
If $\E$ is an egg of $\PG(2n+m-1,q)$ then $m=n$ or $ma=n(a+1)$ with $a$ odd. Moreover, if $q$ even, then $m=n$ or $m=2n$.
\end{theorem}

This explains why the study of eggs is mainly focused on pseudo-ovals and pseudo-ovoids.

In the case of pseudo-ovals, all known examples are elementary. The classical example of an oval in $\PG(2,q^n)$ is a conic. It is a well-known theorem of Segre that an oval of $\PG(2,q^n)$, $q$ odd, is always a conic.
A {\em pseudo-conic} in $\PG(3n-1,q)$ is an elementary pseudo-oval, arising from applying field reduction to a conic in $\PG(2,q^n)$.
We have the following theorems characterising elementary pseudo-ovals using the induced Desarguesian spreads.

\begin{theorem}{\rm \cite{Casse}}\label{Casse}
If $\O$ is a pseudo-oval in $\PG(3n-1,q)$, $q$ odd, such that for at least one element the induced spread is Desarguesian, then $\O$ is a pseudo-conic.
\end{theorem}

\begin{theorem}{\rm \cite{PsO}}
If $\O$ is a pseudo-oval in $\PG(3n-1,q)$, $n$ prime, $q>2$ even, such that all elements induce a Desarguesian spread, then $\O$ is elementary.
\end{theorem}

In the case that $q$ is odd, we have the following theorem which extends Theorem \ref{Casse} from pseudo-ovals to large pseudo-arcs in $\PG(3n-1,q)$.

\begin{theorem}{\rm \cite{Penttila}}
If $\K=\{K_1,\ldots,K_s\}$ is a pseudo-arc in $\PG(3n-1,q)$, $q$ odd, of size greater than the size of the second largest complete arc in $\PG(2,q^n)$,  where for one element $K_i$ of $\K$, the partial spread $\S=\{K_1,\ldots,K_{i-1},K_{i+1},\ldots,K_{s}\}/K_i$ extends to a Desarguesian spread of $\PG(2n-1,q)=\PG(3n-1,q)/K_i$, then $\K$ is contained in a pseudo-conic.
\end{theorem}

In Theorem \ref{main}, we will prove a similar statement for pseudo-caps in $\PG(4n-1,q)$.

All known examples of pseudo-ovoids in $\PG(4n-1,q)$ are elementary when $q$ is even, but in contrast to the situation for pseudo-ovals, when $q$ is odd, there are non-elementary examples of pseudo-ovoids. The standard example of an ovoid in $\PG(3,q^n)$ is an elliptic quadric $Q^{-}(3,q^n)$. By the famous result of Barlotti and Panella \cite{Barlotti,Panella}, every ovoid of $\PG(3,q^n)$, $q$ odd, is an elliptic quadric $Q^{-}(3,q^n)$, however, there is no classification of ovoids in $\PG(3,q^n)$ for $q$ even.
  For both even and odd order $q$, the classification of pseudo-ovoids is an open problem.

\section{Good eggs and Desarguesian spreads}\label{2goodDes}

A (weak) egg $\E$ in $\PG(2n+m-1,q)$, $m>n$, is {\em good} at an element $E \in \E$ if every $(3n-1)$-space containing $E$ and at least two other elements of $\E$, contains exactly $q^n+1$ elements of $\E$. A (weak) egg that has at least one good element is called a {\em good (weak) egg}.
If $\E$ is good at $E$, then for any two elements $E_1,E_2 \in \E \backslash\{E\}$ the $(3n-1)$-space $\langle E, E_1, E_2\rangle$ intersects $\E$ in a pseudo-oval.

\begin{lemma} \label{m=2n}
Good weak eggs in $\PG(2n+m-1,q)$ can only exist if $n$ is a divisor of $m$.
Good eggs only exist in $\PG(4n-1,q)$.
\end{lemma}
\begin{proof}
Consider a weak egg $\E$ of $\PG(2n+m-1,q)$, $m>n$, good at an element $E_1 \in \E$. Consider a second element $E_2 \in \E \backslash\{E_1\}$. For every element $E\in \E\setminus\{E_1,E_2\}$, the $(3n-1)$-space $\langle E, E_1,E_2\rangle$ intersects $\E$ in a pseudo-oval. By considering the elements of $\E\setminus\{E_1,E_2\}$, we find a set $\T$ of  $(3n-1)$-spaces containing $\langle E_1,E_2\rangle$, such that each space of $\T$ intersects $\E$ in a pseudo-oval. Every two spaces in $\T$ meet exactly in $\langle E_1,E_2\rangle$ and $\E$ is the union of the pseudo-ovals $\{T\cap \E|T\in \T\}$. The set $\T$ contains $\frac{q^m-1}{q^n-1}$ $(3n-1)$-spaces; as $q^n-1$ has to be a divisor of $q^m-1$, it follows that $n$ is a divisor of $m$.

Suppose $\E$ is an egg. For $q$ even, by Theorem \ref{mogelijkheden}, eggs only exist in $\PG(4n-1,q)$ (or $\PG(3n-1,q)$).
Consider now a good egg of $\PG(2n+m-1,q)$, $q$ odd, where $m$ is a multiple of $n$. By Theorem \ref{mogelijkheden}, $m=\frac{a+1}{a}n$, for some odd integer $a$, so we conclude that $m=2n$.
\end{proof}

We will show that the good elements of an egg are exactly those inducing a partial spread which is extendable to a Desarguesian spread.
Part $(i)$ of the following theorem, for $\E$ an egg, is mentioned in \cite[Remark 5.1.7]{TGQ}.

\begin{theorem}\label{good=Desarguesian}${}$

\begin{enumerate}

\item[$(i)$] If a weak egg $\E$ in $\PG(2n+m-1,q)$ is good at an element $E$, then $E$ induces a partial spread which extends to a Desarguesian spread.

\item[$(ii)$] Let $\E$ be a weak egg in $\PG(2n+m-1,q)$ for $q$ odd and an egg in $\PG(2n+m-1,q)$ for $q$ even. If an element $E \in \E$ induces a partial spread which extends to a Desarguesian spread, then $\E$ is good at $E$.
\end{enumerate}
\end{theorem}
\begin{proof}
(i) Suppose $\E$ is a weak egg which is good at $E$. Consider the partial spread $\S$ of $\PG(n+m-1,q)$ of size $q^{m}$ induced by $E$. Because $\E$ is good at $E$, any two elements of $\S$ span a $(2n-1)$-space which contains a partial spread of $q^{n}$ elements of $\S$. This partial spread has deficiency 1, so extends uniquely to a spread by one $(n-1)$-space (by \cite{Beutelspacher}).


Consider three elements $S_1, S_2, S_3 \in \S$ not lying in the same $(2n-1)$-space, hence spanning a $(3n-1)$-space $\pi$.
There are $q^n$ elements of $\S$ contained in $\langle S_2, S_3\rangle$.
For every element $R$ of $\S \cap \langle S_2, S_3\rangle$, the $(2n-1)$-space $\langle S_1,R\rangle$ contains $q^n$ elements of $\S$. Hence, there are $q^n$ $(2n-1)$-spaces of $\pi$ containing $S_1$ and $q^n-1$ other elements of $\S$.
Similarly, there are $q^n$ $(2n-1)$-spaces of $\pi$ containing $S_2$ and $q^n-1$ other elements of $\S$.
Since $\pi$ has dimension $3n-1$, two such distinct $(2n-1)$-spaces, one containing $S_1$ and the other containing $S_2$, intersect in at least an $(n-1)$-space, hence, in exactly an $(n-1)$-space. This space is either an element of $\S$ or the $(n-1)$-space which extends both of them to a spread.
It follows that there are $q^{2n}$ elements of $\S$ contained in $\pi$ and if an element of $\S$ intersects $\pi$, then it is contained in $\pi$. Hence, if $\langle S_2, S_3\rangle$ meets a $(2n-1)$-space spanned by $S_1$ and an other element of $\S$, then they meet in an $(n-1)$-space.


As $S_1, S_2, S_3$ were chosen randomly, it follows in general that if two distinct $(2n-1)$-spaces spanned by elements of $\S$ intersect, then they meet in an $(n-1)$-space. They meet either in an $(n-1)$-space of $\S$ or in the $(n-1)$-space which extends the partial spreads of both $(2n-1)$-spaces to a spread.
We see that $\S$ can be uniquely extended to a spread which is normal, thus Desarguesian.

(ii) Now let $\E$ be an egg if $q$ is even and a weak egg if $q$ is odd. Suppose $E$ induces a partial spread $\S$ of size $q^{m}$ which extends to a Desarguesian $(n-1)$-spread $\D$ of $\PG(n+m-1,q)$, hence $m=kn$ for some $k>1$.
There are $\frac{q^m-1}{q^n-1}$ elements of $\D$ not contained in $\S$.

When $\E$ is an egg, the elements of $\D\backslash\S$ span a $(m-1)$-space, corresponding to $T_E$. Hence, any $(2n-1)$-space spanned by two elements of $\S$ contains $q^n$ elements of $\S$ and one element $\D\backslash\S$. So, $\E$ is good at $E$.

Suppose $\E$ is a weak egg, with $q$ odd. As $q$ is odd, no $(3n-1)$-space intersects $\E$ in a pseudo-hyperoval. Hence, any $(3n-1)$-space containing $E$ intersects $\E$ in at most $q^n+1$ elements, so any $(2n-1)$-space spanned by two elements of $\S$ can contain at most $q^n$ elements of $\S$. Hence, any such space must contain at least one element of $\D\backslash\S$.
By field reduction, the elements of the Desarguesian spread $\D$ of $\PG(n+m-1,q)$  are in one-to-one correspondance with the points of $\PG(\frac{m}{n},q^n)$. Any $(2n-1)$-space spanned by two elements of $\D$ must contain at least one element of $\D\backslash\S$. Hence, the points corresponding to $\D\backslash\S$ form a line-blocking set of $\PG(\frac{m}{n},q^n)$. Since $|\D\backslash\S|=\frac{q^m-1}{q^n-1}$, from \cite{BoseBurton} it follows that the points corresponding to $\D\backslash\S$ are the points of a $(\frac{m}{n}-1)$-space, hence the elements of $\D\backslash\S$ span a $(m-1)$-space.
As before, it follows that $\E$ is good at $E$.
\end{proof}

The following corollary, for $\E$ an egg, was also mentioned in \cite[Theorem 4.3.4]{ThasGQ1} in terms of translation generalised quadrangles.

\begin{corollary}\label{odd good= pseudo-conic}
If a weak egg $\E$, $q$ odd, is good at an element $E$, then every pseudo-oval of $\E$ containing $E$ is a pseudo-conic.
\end{corollary}

\begin{proof}
Let $\Pi$ be a $(n+m-1)$-space disjoint from $E$. By Theorem \ref{good=Desarguesian}, the partial spread $\E / E$ in $\Pi$ extends to a Desarguesian spread. Consider a pseudo-oval $\O$ of $\E$ containing $E$. The $q^n$ elements of $\O / E$ are contained in $\E/E$ and thus extend to a Desarguesian spread of the $(2n-1)$-space $\langle\O\rangle\cap\Pi$.

The element $E$ of the pseudo-oval $\O$ induces a partial spread $\O / E$ which extends to a Desarguesian spread, hence, by Theorem \ref{Casse}, the statement follows.
\end{proof}

\section{Characterising good eggs and translation generalised quadrangles of order $(q^n,q^{2n})$}\label{3twogood}
\subsection{Eggs with two good elements}
 An elementary pseudo-ovoid that arises from applying field reduction to an elliptic quadric is called {\em classical}. We recall the following theorem from \cite{TGQ}.
\begin{theorem}{\rm \cite[Theorem 5.1.12
]{TGQ}}

If $q$ is odd and an egg $\E$ in $\PG(4n-1,q)$ has at least two
good elements, then $\E$ is classical.
If $q$ is even and an egg $\E$ in $\PG(4n-1,q)$ has at least four good elements, not
contained in a common pseudo-oval on $\E$, then $\E$ is elementary.
\end{theorem}

It was open problem whether, for $q$ even, being good at two elements is sufficient to be elementary, this was posed as Problem A.5.6 in \cite{TGQ}. We will give an affirmative answer to this question in a more general setting, namely in terms of pseudo-caps. We first need two lemma's concerning Desarguesian spreads.

\begin{lemma}{\rm \cite[Corollary 1.8]{PsO}}\label{regpluseen}
Consider two Desarguesian $(n-1)$-spreads $\S_1$ and $\S_2$ in $\PG(2n-1,q)$, $q>2$. If $\S_1$ and $\S_2$ have at least $3$ elements in common, then they share exactly $q^t+1$ elements for some $t|n$.
\end{lemma}
The following lemma is a generalisation of \cite[Lemma 1.4]{PsO} and the proof is analogous. We introduce some necessary definitions and notations.

A {\em regulus} $\R$ in $\PG(2n-1, q)$ is a set of $q + 1$ mutually
disjoint $(n - 1)$-spaces having the property that if a line meets 3 elements
of $\R$, then it meets all elements of $\R$. Let us denote the unique regulus through $3$ mutually disjoint $(n-1)$-spaces $A,B$ and $C$ in $\PG(2n-1,q)$ by $\R(A,B,C)$.
Every Desarguesian spread $\D$ has the property that for $3$ elements $A,B,C$ in $\D$, the elements of $\R(A,B,C)$ are also contained in $\D$, i.e. $\D$ is regular (see also \cite{Br}).

We will use the following notation for points of a projective space $\PG(r-1,q^n)$. A point $P$ of $\PG(r-1,q^n)$ defined by a vector $(x_1,x_2,\ldots,x_r)\in (\F_{q^n})^r$ is denoted by $\F_{q^n}(x_1,x_2,\ldots,x_r)$, reflecting the fact that every $\F_{q^n}$-multiple of $(x_1,x_2,\ldots,x_r)$ gives rise to the point $P$.
We can identify the vector space $\F_{q^{nr}}$ with $(\F_{q^n})^r$, and hence write every point of $\PG(rn-1,q)$ as $\F_q(x_1,\ldots,x_r)$ with $x_i \in \F_{q^n}$. In this way, by field reduction, a point $\F_{q^n}(x_1,\ldots,x_r)$ in $\PG(r-1,q^n)$ corresponds to the $(n-1)$-space  $\F_{q^n}(x_1,\ldots,x_r) = \{  \F_{q}(\alpha x_1,\ldots,\alpha x_r) | \alpha \in \F_{q^n}\}$ of $\PG(rn-1,q)$.

\begin{lemma}\label{UniqueDesSpread}
Let $\D_1$ be a Desarguesian $(n-1)$-spread in a $(kn-1)$-dimensional subspace $\Pi$ of $\PG((k+1)n-1,q)$, let $\mu$ be an element of $\D_1$ and let $E_1$ and $E_2$ be mutually disjoint $(n-1)$-spaces such that $\langle E_1,E_2\rangle$ meets $\Pi$ exactly in the space $\mu$. Then there exists a unique Desarguesian $(n-1)$-spread of $\PG((k+1)n-1,q)$ containing $E_1$, $E_2$ and all elements of $\D_1$.
\end{lemma}
\begin{proof} Since $\D_1$ is a Desarguesian spread in $\Pi$, we can choose coordinates for $\Pi$ such that $\D_1=\{\F_{q^n}(x_1,x_2,\ldots,x_{k})|x_i\in \F_{q^n}\}$ and $\mu=\F_{q^n}(0,\ldots,0,1)$. We embed $\Pi$ in $\PG((k+1)n-1,q)$ by mapping a point $\F_q(x_1,\ldots,x_{k})$, $x_i\in \F_{q^n}$, of $\Pi$ onto $\F_q(x_1,\ldots,x_{k},0)$. Let $\ell_P$ denote the unique transversal line through a point $P$ of $\mu$ to the regulus $\R(\mu,E_1,E_2)$.

We can still choose coordinates for $n+1$ points in general position in $\PG((k+1)n-1,q)\setminus \Pi$. We will choose these $n+1$ points such that $n$ of them belong to $E_1$ and one of them belongs to $E_2$.
Consider a set $\{y_i|i=1,\ldots,n\}$ forming a basis of $\F_{q^n}$ over $\F_q$. We may assume that the line $\ell_{P_i}$ through $P_i=\F_q(0,\ldots,0,y_i,0)\in\mu$ meets $E_1$ in the point $\F_q(0,\ldots,0,0,y_i)$. It follows that $E_1=\F_{q^n}(0,\ldots,0,0,1)$. Moreover, we may assume that $\ell_Q$ with $Q=\F_q(0,\ldots,0,0,\sum_{i=1}^n y_i,0)\in \mu$ meets $E_2$ in $\F_q(0,\ldots,0,\sum_{i=1}^n y_i,\sum_{i=1}^n y_i)$. Since $\F_q(0,\ldots,0\sum_{i=1}^n y_i,\sum_{i=1}^n y_i)$ has to be in the space spanned by the intersection points $R_i = \ell_{P_i} \cap E_2$, it follows that $R_i=\F_q(0,\ldots,0,y_i,y_i)$ and consequently, that $E_2=\F_{q^n}(0,\ldots,0,1,1)$.

It is clear that the Desarguesian spread $\D=\{\F_{q^n}(x_1,\ldots,x_{k+1})|x_i\in \F_{q^n}\}$ contains the spread $\D_1$ and the $(n-1)$-spaces $E_1$ and $E_2$. Moreover, since every element of $\D$, not in $\langle E_1,E_2\rangle$, is obtained as the intersection of $\langle E_1,X\rangle\cap \langle E_2,Y\rangle$, where $X,Y\in \D_1$, it is clear that $\D$ is the unique Desarguesian spread satisfying our hypothesis.
\end{proof}

\begin{lemma}\label{disjoint}
Consider a pseudo-cap $\E$ of $\PG(4n-1,q)$ containing an element $E$ that induces a partial spread which extends to a Desarguesian spread. If $\Pi$ is a $(3n-1)$-space spanned by $E$ and two other elements of $\E$, then every element of $\E$ is either disjoint from $\Pi$ or contained in $\Pi$.
\end{lemma}

\begin{proof} Let $\Sigma$ be a $(3n-1)$-space skew from $E$ and consider the induced partial spread $\E/E$ in $\Sigma$. If $F$ is an element of $\E$ which meets $\Pi$, then the projection $F/E$ of $F$ from $E$ onto $\Sigma$ is an element of $\E/E$ which meets the space $\Pi/E$. By assumption, the space $\Pi/E$ is spanned by spread elements of a partial spread extending to a Desarguesian spread. Hence, since a Desarguesian spread is normal, $F/E$ is contained in $\Pi/E$. It follows that since $\Pi$ contains $E$, the element $F$ is contained in $\Pi$. \end{proof}

\begin{theorem}\label{main}
Consider a pseudo-cap $\E$ in $\PG(4n-1,q)$, $q>2$, with $|\E|>q^{n+k}+q^n-q^k+1$, $q$ odd, and $|\E|>q^{n+k}+q^n+2$, $q$ even, where $k$ is the largest divisor of $n$ with $k\neq n$.
 The pseudo-cap $\E$ is elementary if and only if two of its  elements induce a partial spread which extends to a Desarguesian spread.
\end{theorem}
\begin{proof}
If $\E$ is elementary, then the elements of $\E$ are contained in a Desarguesian spread of $\PG(4n-1,q)$, so every element of $\E$ induces a partial spread which extends to a Desarguesian spread.

Now suppose that $\E$ contains two distinct elements $E_1,E_2$ that induce a partial spread which extends to a Desarguesian spread.
Since $|\E| > q^{n}+2$, using Lemma \ref{disjoint}, we can find two elements $E_3, E_4 \in \E$ such that $\langle E_1, E_2, E_3, E_4\rangle$ spans $\PG(4n-1,q)$.

The partial spread induced by $E_1$ in the space $\langle E_2,E_3,E_4\rangle$ can be extended to a Desarguesian spread $\D_1$.
Analogously, the partial spread induced by $E_2$ in the space $\langle E_1,E_3,E_4\rangle$ can be extended to a Desarguesian spread $\D_2$. Since $E_3$ and $E_4$ are elements of the spreads $\D_1$ and $\D_2$, the Desarguesian spreads $\D_1$ and $\D_2$ intersect the $(2n-1)$-space $\langle E_3, E_4\rangle$ each in a Desarguesian spread, say $\S_1$ and $\S_2$ respectively.

Take an element $E \in \E \backslash\{E_1,E_2\}$ and consider the $(3n-1)$-space $\langle E_1,E_2,E\rangle$. From Lemma \ref{disjoint} it follows that any element of $\E$ is either contained in or disjoint from $\langle E_1,E_2,E\rangle$. By considering the elements of $\E\setminus\{E_1,E_2\}$, we find a set $\T$ of $(3n-1)$-spaces containing $\langle E_1,E_2\rangle$, such that each space of $\T$ intersects $\E$ in a pseudo-arc. Every two spaces in $\T$ meet exactly in $\langle E_1,E_2\rangle$ and $\E$ is the union of the  pseudo-arcs $\{T\cap \E|T\in \T\}$.
The set $\T$ intersects $\langle E_3, E_4\rangle$ in a partial $(n-1)$-spread $\P$.

Let $P$ be an element of $\P$, then $\langle P,E_1,E_2\rangle$ is a $(3n-1)$-space containing at least one element $E$ of $\E\backslash\{E_1,E_2\}$. The projection $E'$ of $E$ from $E_1$ onto $\langle E_2,E_3,E_4\rangle$ is contained in $\D_1$. We obtain that $P=\langle E_2,E'\rangle \cap \langle E_3,E_4\rangle$, and since the elements $E_2,E',E_3,E_4$ are contained in $\D_1$, this implies that $P$ is contained in $\D_1$. Moreover, since $P\subset \langle E_3,E_4\rangle$, the element $P$ is contained in $\S_1$. Similarly, we obtain that $P$ is contained in $\S_2$ and we conclude that every element of $\P$ must be contained in both $\S_1$ and $\S_2$.

Suppose that $k$ is the largest divisor of $n$ with $k\neq n$. The pseudo-cap $\E$ has size $|\E|> (q^{n}-\epsilon)(q^k+1)+2$ and every $(3n-1)$-space of $\T$ contains at most $q^n-\epsilon$ elements different from $E_1, E_2$, where $\epsilon=1$ for $q$ odd and $\epsilon=0$ for $q$ even. By the pigeonhole principle, it follows that $|\P|\geq q^k+2$.  Hence, the Desarguesian spreads $\S_1$ and $\S_2$ have at least $q^k+2$ elements in common, where $k$ is the largest divisor of $n$ with $k\neq n$. As $q>2$, by Lemma \ref{regpluseen}, we find that $\S_1=\S_2$.

By Theorem \ref{UniqueDesSpread}, consider the unique Desarguesian spread $\D$ of $\PG(4n-1,q)$ containing all elements of $\D_1$ and two distinct elements of $\D_2 \backslash \D_1$. It is clear that, since $\S_1=\S_2$, the spread $\D$ contains all elements of $\D_2$.

Every element of $\E$, not in $\D_1 \cup \D_2$, arises as the intersection $\langle E_1, X \rangle \cap \langle E_2, Y \rangle$ for some $X \in \D_1 \subset \D$ and $Y \in \D_2 \subset \D$, hence, since a Desarguesian spread is normal, every element of $\E$ belongs to $\D$. It follows that $\E$ is elementary.
\end{proof}

We obtain the following corollary which improves \cite[Theorem 5.1.12]{TGQ}.
\begin{corollary}
A weak egg in $\PG(4n-1,q)$ which is good at two distinct elements is elementary.
\end{corollary}
\begin{proof}
A weak egg is a pseudo-cap of size $q^{2n}+1$ in $\PG(4n-1,q)$. By Theorem \ref{good=Desarguesian}, if the weak egg is good at two elements, these elements induce a partial spread which extends to a Desarguesian spread. We can repeat the proof of Theorem \ref{main}. Now the partial spread $\P$ has size $q^n+1$, so the conclusion $\S_1=\S_2$ follows immediately. We do not require Lemma \ref{regpluseen}, hence the restriction $q>2$ can be dropped.
\end{proof}

\subsection{A corollary in terms of translation generalised quadrangles}\label{3.2 TGQ}
A {\em generalised quadrangle} of order $(s,t)$, $s,t>1$, is an incidence structure of points and lines satisfying the following axioms:
\begin{itemize}
\item every line has exactly $s+1$ points,
\item through every point, there are exactly $t+1$ lines,
\item if $P$ is a point, not on the line $L$, then there is exactly one line through $P$ which meets $L$ non-trivially.
\end{itemize}

From every {\em egg} $\E$ in $\Sigma_{\infty}=\PG(2n+m-1,q)$ we can construct a generalised quadrangle $(\P,\L)$ as follows. Embed $\Sigma_{\infty}$ as a hyperplane at infinity of $\PG(2n+m,q)$.
\begin{itemize}
\item[$\P$:]
\begin{itemize}
\item[$(i)$] {\em affine} points of $\PG(2n+m,q)$, i.e. the points not lying in $\Sigma_\infty$,
\item[$(ii)$] the $(n+m)$-spaces meeting $\Sigma_\infty$ in $T_E$ for some $E\in \E$,
\item[$(iii)$] the symbol $(\infty)$.
\end{itemize}
\item[$\L$:]
\begin{itemize}
\item[$(a)$] the $n$-spaces meeting $\Sigma_\infty$ in an element of $\E$,
\item[$(b)$] the elements of $\E$.
\end{itemize}
\end{itemize}
Incidence is defined as follows.
\begin{itemize}
\item A point of type $(i)$ is incident with the lines of type $(a)$ through it.
\item A point of type $(ii)$ is incident with the lines of type $(a)$ it contains and the line of type $(b)$ it contains.
\item The point $(\infty)$ is incident with all lines of type $(b)$.
\end{itemize}
The obtained generalised quadrangle is denoted as $T(\E)$ and is called a {\em translation generalised quadrangle} (TGQ) with base point $(\infty)$. In \cite[Theorem 8.7.1]{FGQ} it is proven that every TGQ of order $(q^n,q^{m})$, where $\F_q$ is a subfield of its kernel, is isomorphic to a $T(\E)$ for some egg $\E$ of $\PG(2n+m-1,q)$.

When $n=m=1$, then $\O$ is an oval of $\PG(2,q)$ and the construction above gives the well-known construction of $T_2(\O)$. When $n=1$ and $m=2$, then $\O$ is an ovoid of $\PG(3,q)$ and the construction above is the construction of Tits of $T_3(\O)$ (see \cite{TGQ}).

\begin{lemma}  \label{TGQ1} Let $T=T(\E)$ be a TGQ of order $(q^n,q^{2n})$ with base point $(\infty)$.
Let $m_1,m_2, m_3$ be three distinct lines through $(\infty)$, and let $E_1$, $E_{2}$, $E_{3}$ denote the elements of $\E$ corresponding to $m_1,m_2,m_3$ respectively. Then there is a subquadrangle of order $q^n$ through $m_1,m_2,m_3$ if and only if the $(3n-1)$-dimensional space $\langle E_1,E_{2},E_{3}\rangle$ contains exactly $q^n+1$ elements of $\E$.
\end{lemma}

\begin{proof} Suppose that the $(3n-1)$-space $\Sigma=\langle E_1,E_{2},E_{3}\rangle$ contains a set $\O$ of exactly $q^n+1$ elements of $\E$, then it is clear that $T(\E)$ defines the incidence structure $T(\O)$ in a $3n$-space through $\Sigma$. The structure $T(\O)$ is a generalised quadrangle of order $q^n$, forming a subquadrangle of $T(\E)$ and containing the lines $m_1,m_2,m_3$.

On the other hand, suppose that there is a subquadrangle $T'$ of order $q^n$ containing $m_1,m_2,m_3$, where the lines $m_1,m_2,m_3$ are incident with $(\infty)$. This implies that the point $(\infty)$ is in $T'$, and since $(\infty)$ lies only on lines of type ($b$) (i.e. the lines corresponding to elements of $\E$), we know that $T'$ contains exactly $q^n+1$ lines of type $(b)$, among which the lines $m_1, m_2$ and $m_3$. Let $\{E_1,\ldots,E_{q^n+1}\}$ be the egg elements corresponding to these lines. This means that there are $(q^n+1)q^{2n}$ lines in $T'$ of type $(a)$, containing in total $(q^{n}+1)q^{2n}(q^{n})/(q^n+1)=q^{3n}$ points of type $(i)$ (i.e. affine points).

Each $(n-1)$-space $E_j$ is contained in $q^{2n}$ $n$-spaces corresponding to a line of type $(a)$ of $T'$ and every affine point is contained in exactly one $n$-space containing $E_j$.
Let $P_j$ be a point of the space $E_j$, then we see that the $q^{3n}$ affine points of $T'$ lie on $q^{2n}$ lines through $P_j$. As this holds for every $j\in \{1,\ldots,q^n+1\}$, it is clear that the $q^{3n}$ affine points of $T'$ are contained in a $3n$-space. This in turn implies that the elements $\{E_1,\ldots,E_{q^n+1}\}$ are contained in a $(3n-1)$-space, namely $\langle E_1,E_2,E_{3}\rangle$. Hence, this space contains at least $q^n+1$ elements of $\E$. Since $\E$ is an egg, it is not possible that a $(3n-1)$-space contains more than $q^n+1$ elements of $\E$, which concludes the proof.
\end{proof}

\begin{lemma}Let $T=T(\E)$ be a TGQ of order $(q^n,q^{2n})$ with base point $(\infty)$. Let $\ell$ be a line through $(\infty)$ and $E_\ell$ the element of $\E$ corresponding to $\ell$. The egg $\E$ is good at $E_\ell$ if and only if for every two distinct lines $m_1,m_2$ through $(\infty)$, where $m_1,m_2\neq \ell$, there is a subquadrangle of order $q^n$ through $m_1,m_2, \ell$.
\end{lemma}
\begin{proof} This follows immediately from Lemma \ref{TGQ1} and the definition of a being good at an element.
\end{proof}

We are now ready to state the promised characterisation of the translation generalised quadrangle $T_3(\O)$.
\begin{theorem} Let $T$ be a TGQ of order $(q^n,q^{2n})$ with base point $(\infty)$.
Suppose that $T$ contains two distinct lines $\ell_i$, $i=1,2$ such that for every two distinct lines $m_1,m_2$ through $(\infty)$, where $m_1,m_2\neq \ell_i$, $i=1,2$ there is a subquadrangle through $m_1,m_2,\ell_i$, $i=1,2$, then $T$ is isomorphic to $T_3(\O)$, where $\O$ is an ovoid of $\PG(3,q^n)$.
\end{theorem}

\section{A geometric proof of a Theorem of Lavrauw}\label{4ThmLavrauw}

In this section we obtain a second characterisation of good weak eggs. We need the following lemma stating that every good element of a weak egg has a tangent space.

\begin{lemma}\label{tangentspace}
If a weak egg $\E$ in $\PG(2n+m-1,q)$ is good at an element $E$, then there exists a unique $(n+m-1)$-space $T$, such that $T \cap \E = \{E\}$.
\end{lemma}
\begin{proof}
Consider a $(n+m-1)$-space $\Sigma$ disjoint from $E$.
If $\E$ is good at $E$, the element $E$ induces a partial spread $\S=\E / E$ which extends to a Desarguesian spread $\D$ of $\Sigma$. By following the proof of Theorem \ref{good=Desarguesian}, part $(ii)$, for both $q$ odd and $q$ even, the elements of $\D \backslash \S$ span a $(m-1)$-space. It is clear that the $(n+m-1)$-space $T =\langle E, \D\backslash\S\rangle$ satisfies $T \cap \E =E$.
\end{proof}

In \cite{LavrauwPenttila} the authors proved that every egg of $\PG(7,2)$ arises from an elliptic quadric $Q^-(3,4)$ by field reduction. Hence, in the following characterisation, when $\E$ is an egg in $\PG(4n-1,q)$, the condition $q^n>4$ is essentially not a restriction.

\begin{theorem}
Suppose $n>1$, $q^n>4$, consider $\E$ a weak egg in $\PG(4n-1,q)$. Then $\E$ is elementary if and only if the following three properties hold:
\begin{itemize}
\item $\E$ is good at an element $E$,
\item there exists a $(3n-1)$-space, disjoint from $E$, with at least $5$ elements $E_1$, $E_2$, $E_3$, $E_4$, $E_5$ of $\E$,
\item all pseudo-ovals of $\E$ containing $\{E,E_1\}$, $\{E,E_2\}$ or $\{E,E_3\}$ are elementary.
\end{itemize}
\end{theorem}
\begin{proof}
Clearly, if an egg is elementary, the statement is valid.

For the converse, consider the $(3n-1)$-space $\Pi$ containing 5 elements $E_1,E_2,E_3,E_4,E_5$ of $\E$, but not the element $E$. As $\E$ is good at $E$, the element $E$ induces a partial spread which extends to a Desarguesian $(n-1)$-spread $\D_0$ in $\Pi$, which contains $E_i$, $i=1,\ldots,5$.

By Lemma \ref{tangentspace}, there exists a unique $(3n-1)$-space $T$, such that $T \cap \E=\{E\}$. When $\E$ is an egg, this space corresponds to the tangent space $T_E$.

Consider the two $(n-1)$-spaces $F=\langle E_1, E_5 \rangle \cap \langle E_2, E_4 \rangle$ and $F'=\langle E_1, E_5\rangle \cap \langle E_3, E_4 \rangle$. Both $F$ and $F'$ are contained in $\D_0$, but at most one of them can be contained in the $(2n-1)$-space $\Pi\cap T$. Suppose $F$ is not contained in $T$ (note that this choice has no further impact as $E_2$ and $E_3$ play the same role). This implies that the $(2n-1)$-space $\langle E, F \rangle$ contains an element $E_6 \in \E \backslash \{E\}$. By Theorem \ref{UniqueDesSpread}, there exists a unique Desarguesian spread $\D$ containing $E$, $E_6$ and all elements of $\D_0$. We will prove that $\E$ is contained in $\D$.

The $(3n-1)$-space $\langle E, E_1,E_5\rangle$ intersect $\E$ in a pseudo-oval $\O_1$, and the $(3n-1)$-space $\langle E, E_2,E_4\rangle$ intersect $\E$ in a pseudo-oval $\O_2$. Clearly, $\O_1$ and $\O_2$ both contain $E_6$.

By assumption, $\O_1$ and $\O_2$ are elementary pseudo-ovals. The Desarguesian $(n-1)$-spread in $\langle E,E_1,E_5 \rangle$ containing $\O_1$ contains $E$, $E_6$ and the $q^n+1$ elements of ${\D_0}  \cap \langle E_1, E_5\rangle$. It follows that this Desarguesian spread is contained in $\D$, hence $\O_1$ is contained in $\D$. Analogously, the pseudo-oval $\O_2$ is also contained in $\D$.

There are $q^n-2$ pseudo-ovals $\O$ of $\E$, containing $\{E,E_3\}$, but not $E_6$, such that the $(3n-1)$-space $\langle \O\rangle$ does not contain the $(n-1)$-space $T \cap \langle \O_1\rangle$, nor the $(n-1)$-space $T \cap \langle \O_2\rangle$. Take such an oval $\O$, then there is an element $E_7$ of $\E\setminus \{E\}$ contained in  $\langle \O\rangle\cap \langle\O_1\rangle$, hence, $E_7\in \O\cap \O_1$. Likewise, there is an element $E_8$ of $\E\setminus \{E\}$ contained in $\O\cap \O_2$.

By assumption, $\O$ is elementary; let $\S_{\O}$ be the Desarguesian $(n-1)$-spread containing $\O$. As $E_7$ and $E_8$ are contained in $\D$, the Desarguesian spread $\D$ intersects $\langle E_7,E_8\rangle$ in a Desarguesian spread. Let $P$ be an element of $\D\cap \langle E_7,E_8\rangle$, not contained in $T$, then $\langle E,P\rangle$ meets $\Pi$ in an element of $\D$, and hence, $\langle E,P\rangle$ contains an element $P'$ of $\E\setminus E$. As $\langle E,P\rangle$ is contained in $\langle \O\rangle$, $P'$ is an element of $\O$, and hence also of $\S_{\O}$. Since $P',E,E_7,E_8$ are contained in $\S_\O$, the element $P=\langle E,P'\rangle\cap \langle E_7,E_8\rangle$ is an element of $\S_{\O}$. This implies that $\D\cap \langle E_7,E_8\rangle$ and $\S_{\O}$ have at least $q^n$ elements in common, which implies in turn that they have all their elements in common.
We conclude that $\S_{\O}$ contains $E$, $E_3$ and the $q^n+1$ elements of $\D \cap \langle E_7, E_8\rangle$, hence $\S_\O$ and thus all elements of $\O$ are contained in $\D$.

Now, consider an element $E_9 \in \E$, not contained in $\O_1$, $\O_2$ or any of the previously considered $q^n-2$ pseudo-ovals $\O$. Look at the pseudo-oval $\O'=\langle E,E_1,E_9\rangle\cap \E$ and the pseudo-oval $\O''=\langle E,E_2,E_9\rangle \cap \E$. At least one of them does not contain $E_3$. Suppose $\O'$ does not contain $E_3$ (the proof goes analogously if $\O''$ does not contain $E_3$).
For at most one of the $q^n-2$ pseudo-ovals $\O$ containing $\{E,E_3\}$ we have $\langle \O\rangle \cap \langle \O''\rangle \in T$. Hence, since $q^n-2\geq 3$, we can find two distinct elementary pseudo-ovals containing $\{E,E_3\}$ that are contained in $\D$ and have an element $E_{10}$ and $E_{11}$ respectively in common with $\O'$.

Let $\S_{\O'}$ be the Desarguesian $(n-1)$-spread containing $\O'$. As $E_{10}$ and $E_{11}$ are elements of $\D$ the same argument as before shows that all but one element of the Desarguesian spread $\D\cap \langle E_{10},E_{11}\rangle$ can be written as the intersection of $\langle E,P''\rangle$ with $\langle E_{10},E_{11}\rangle$ for some $P''$ in $\O'$. It follows that $\S_{\O'}$ contains $E,E_{1}$ and the $q^n+1$ elements of $\D\cap \langle E_{10},E_{11}\rangle$, hence, that $\S_{\O'}$ is contained in $\D$. In particular, the element $E_9$ is contained in $\D$, which implies that $\E\subset \D$ and so that $\E$ is elementary and more specifically, a field reduced ovoid.
\end{proof}

When $\E$ is good at $E$ and $q$ is odd, by Corollary \ref{odd good= pseudo-conic} all pseudo-ovals of $\E$ containing $E$ are pseudo-conics; we use this to obtain the following corollary. The same statement, where $\E$ is an egg, was proven in \cite[Theorem 3.2]{LavrauwEggs} using coordinates. For $\E$ an egg, this was also shown in \cite[Theorem 5.2.3]{TGQ} where a different proof was obtained independently, relying on a technical theorem concerning the $\F_{q^n}$-extension of the egg elements. We have now obtained a direct geometric proof.

\begin{corollary}
A weak egg $\E$ of $\PG(4n-1,q)$, $q$ odd, $n>1$, is classical if and only if it is good at an element $E$ and there exists  a $(3n-1)$-space, not containing $E$, with at least $5$ elements of $\E$.
\end{corollary}

{\bf Acknowledgment.} The authors wish to thank Simeon Ball for suggesting the study of eggs in terms of the induced (partial) spreads.

\end{document}